\documentclass[11pt]{amsart}
\usepackage{url,enumitem, amsopn, amsmath}
\usepackage{xcolor}
\usepackage{cleveref}
\usepackage{tikz}
\usepackage{wrapfig}
\usepackage{lipsum}
\usepackage{pgfplots}
\usepackage{filehook}
\AtEndOfIncludes{%
  \global\let\savedclearpage\clearpage
  \global\let\clearpage\relax
}
\AfterIncludes{%
  \global\let\clearpage\savedclearpage
}
\pgfdeclarelayer{nodelayer}
\pgfdeclarelayer{edgelayer}
\pgfsetlayers{nodelayer,main,edgelayer}

\newcommand{\cal}{\mathcal}

\newcommand{\reals}{\mbox{$\mathbb R$}}

\newcommand{\nats}{\mbox{$\mathbb N$}}
\newcommand{\ints}{\mbox{$\mathbb Z$}}

\newcommand{\floor}[1]{\lfloor #1 \rfloor}

\newcommand{\odim}{{\mathop{\mathrm{odim}}\nolimits}}

\newcommand{\inc}{{\mathop{\mathrm{inc}}\nolimits}}
\newcommand{\crit}{{\mathop{\mathrm{crit}}\nolimits}}
\newcommand{\extr}{{\mathop{\mathrm{extr}}\nolimits}}
\newcommand{\tr}{{\mathop{\mathrm{tr}}\nolimits}}

\newcommand{\Pos}{{\mathop{\mathbf{P}}\nolimits}}
\newcommand{\chain}{{\mathop{\mathbf{C}}\nolimits}}
\newcommand{\Turan}{{\mathop{\mathbf{T}}\nolimits}}

\newcommand{\comment}[1]{}

\usepackage{comment}

\usepackage{amsmath, amsfonts, amssymb, latexsym, stmaryrd}
\usepackage{graphicx}                   

\def\squarebox#1{\hbox to #1{\hfill\vbox to #1{\vfill}}}
\def\qed{\hspace*{\fill}
        \vbox{\hrule\hbox{\vrule\squarebox{.667em}\vrule}\hrule}\smallskip}

\theoremstyle{plain}
\newtheorem{lemma}{Lemma}[section]
\newtheorem{theorem}[lemma]{Theorem}
\newtheorem{corollary}[lemma]{Corollary}
\newtheorem{proposition}[lemma]{Proposition}

\theoremstyle{definition}
\newtheorem{claim}[lemma]{Claim}
\newtheorem{observation}[lemma]{Observation}
\newtheorem{definition}[lemma]{Definition}

\newtheorem{example}[lemma]{Example}

\newtheorem*{rmk*}{Remark}
\newtheorem*{rmks*}{Remarks}
\newtheorem*{conventions*}{Conventions}
\newtheorem*{convention*}{Convention}

\def\squareforqed{\hbox{\rlap{$\sqcap$}$\sqcup$}}
\def\qed{\ifmmode\squareforqed\else{\unskip\nobreak\hfil
\penalty50\hskip1em\null\nobreak\hfil\squareforqed
\parfillskip=0pt\finalhyphendemerits=0\endgraf}\fi}

\newlength{\tablength}
\newlength{\spacelength}
\newcommand{\tabstar}{\hspace*{\tablength}}
\newcommand{\spacestar}{\hspace*{\spacelength}}
\def\obeytabs{\catcode`\^^I=\active}
{\obeytabs\global\let^^I=\tabstar}
{\obeyspaces\global\let =\spacestar}
\newenvironment{display}{\begingroup\obeylines\obeyspaces\obeytabs}{\endgroup}
\newenvironment{prog}{\begin{display}\parskip0pt\sf}{\end{display}}

\author{Geir Agnarsson}
\address{Department of Mathematical Sciences \\ 
George Mason University \\ Fairfax, VA  22030}
\email{geir@math.gmu.edu}

\author{John B.~Kent}
\address{Department of Mathematical Sciences \\ 
George Mason University \\ Fairfax, VA  22030}
\email{johnbkentphd@gmail.com}

\title{Tur\'{a}n results for posets and their alternating cycles}



\date{\today}
\begin{document}
\begin{abstract}
   For a partially ordered set ${\Pos} = (X,\leq)$ there exist
  hypergraphs where the vertices are the set of ordered tuples of
  either all incomparable elements of ${\Pos}$ or all the critical
  pairs of ${\Pos}$, and the edges are formed by the duals of either all
  the alternating cycles of ${\Pos}$ or all the strict alternating
  cycles of ${\Pos}$. The weak chromatic numbers of these hypergraphs are all equal to the order dimension of $\Pos$. Here are established upper bounds on the number of strict alternating cycles a poset $\Pos=(X,\leq)$
 can have in terms of $n = |X|$, the cardinality of the groundset of ${\Pos}$, and the width $w$ of $\Pos$. These bounds also apply to the number of hyperedges of the associated
  hypergraph $\mathcal{H}^s(\Pos)$, with incomparable pairs as vertices and strict alternating cycles dual to its hyperedges.  
\end{abstract}
\maketitle  


\section{Introduction and Motivation}
\label{sec:intro}

Alternating cycles have been employed in a variety of ways in the study of posets over the last several decades. First defined in \cite{TrotterMoore77}, the fundamental link between alternating cycles and order dimension has been used to establish many results bounding order dimension. In the inciting work, it was shown using alternating cycles that introduction of maximum or minimum elements restricts order dimension of planar posets to at most $3$. The wonderful textbook \cite{Trotter}, from which many definitions and notations in the first two sections are derived, leverages these structures to show a wide variety of enumerative results for strict alternating cycles. As an example, alternating cycles in posets have been used to show the following theorem from~\cite{trotter1976some}:
\begin{theorem} 
    Let $G=(V,E)$ be a connected graph on two or more vertices. Let $\Pos=(X,\leq)$ be a poset where $X$  is the set of all connected induced subraphs of $G$ and $H_1\leq H_2$ in $\Pos$ iff $H_1$ is contained in $H_2$. Then the dimension of $\Pos$ is the number of noncut vertices in $G$.
\end{theorem}
More recently, alternating cycles have been used to show planar posets with height $h$ have dimension at most $192h+96$ \cite{joret2017planar}, and that planar posets of bounded height whose cover graphs  exclude a fixed topological minor have bounded dimension \cite{walczak2017minors}. A cursory glance at how alternating cycles can be generated will reveal very quickly that the number of alternating cycles possible on a poset is at least exponential in its cardinality. Specificity beyond this was heretofore unexplored, and so establishing tighter bounds is the goal of this work. 

\section{Setup, Definitions, and Some Classical Results}
Before commencing the further discussion, we state some conventions and basic definitions 
we will use throughout the article. Essentially all definitions will be consistent with \cite{Trotter}.
A \textit{partial order} $\leq$ is a binary relation on a set $X$ that is reflexive ($a\leq a$), antisymmetric ($a\leq b, ~~ b\leq a \implies a=b$), and transitive ($a\leq b, ~~b\leq c \implies a \leq c$) for all $a,b,c \in X$.
A \textit{partially ordered set}, or \textit{poset}, $\Pos=(X,\leq)$ is a set $X$ equipped with a partial order $\leq$. The \textit{order} of the poset $\Pos$ is $n=|X|$, the cardinality of the ground set $X$. Unless otherwise noted, from here on out all posets will be of finite order; $|X|=n \in \{1,2,\dots,\}:=\nats$. 
If $x\leq y$ or $y\leq x$, then $x$ and $y$ are called \textit{comparable}. If not, they are \textit{incomparable}, denoted $x||y$. 
For posets $\Pos=(X,\leq)$ and $\Pos'=(X',\leq')$, we say that $\Pos'$ is a \textit{subposet} of $\Pos$, denoted $\Pos' \subseteq \Pos$, if $X' \subseteq X$ and $\leq' \hspace{2pt}\subseteq\hspace{2pt}\leq$ (meaning $x \leq' y \implies x \leq y$ for all $x,y \in X'$).
If $\Pos'\subseteq \Pos$ is a subposet and $\leq'$ is the restriction of $\leq$ to $X'$, then $\Pos'$ is the subposet \textit{induced by $X'$} in $\Pos$, denoted $\Pos'=\Pos\left[X'\right]$. If $\Pos' \subseteq \Pos$ is a subposet and $X'=X$, such that $x \leq' y \implies x \leq y$ for all $x,y \in X$, then $\Pos$ is an \textit{extension} of $\Pos'$. A \textit{chain} $\chain$ of $\Pos$ is a subposet where every two elements are comparable. The \textit{length} of a chain is one less than the number of elements in
the groundset of $\chain$. A poset that is also a chain is a \textit{totally ordered set}, and in that case $\leq$ is a \textit{total order} or \textit{linear order}. An extension that is a total order is called a \textit{linear extension}.
An \textit{antichain} is a subposet where no elements are comparable. A poset that is itself 
an antichain will be called \textit{trivial}. 
The \textit{height} of a poset $\Pos$ is the number of elements of the longest chain. The \textit{width} of $\Pos$ is the number of elements of the largest antichain.
The set $\inc(\Pos)$ is the set of all ordered incomparable pairs $(x,y)$ where $x||y$ in $\Pos$.
For a poset $\Pos$, $(x,y) \in \inc(\Pos)$ is a \textit{critical pair} if for all $z,w \in \Pos$
we have
\begin{enumerate}[label=\roman*]
    \item $z\leq x \implies z \leq y$,
    \item $y \leq w \implies x \leq w$.
\end{enumerate} 
The set of all critical pairs of $\Pos$ is denoted $\crit(\Pos)$.
The \textit{dual} $A^d$ of a subset $A \subseteq X \times X=X^2$ is given by $A^d=\{(b,a):(a,b) \in A\}$. Clearly we have that $(A^d)^d=A$. 
For a binary relation $R$ on set $X$, the \textit{transitive closure} of $R$, denoted ${\tr}(R)$, is
given by 
$$\{(x,y) \in X\times X: \text{there exists a sequence}~~u_0,u_1,\dots,u_n~~ \text{with}~~(u_i,u_{i+1}) \in R, u_0=x,u_n=y\}.$$
Of particular interest of the previous definitions are the critical pairs. Intuitively, one can think of these as pairs that are ``almost'' comparable, in the sense that to the rest of the poset it may as well be the case that $x\leq y$. 
\begin{example}
Figure \ref{fig:critpairs} is a Hasse diagram of a poset.
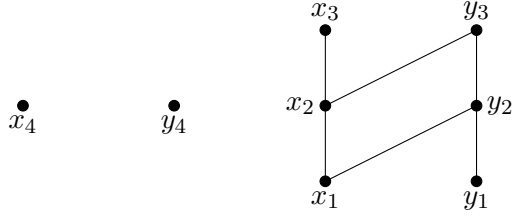
\begin{figure}[h!]
    \centering
    \begin{tikzpicture}
        \draw (-1,-1) -- (-1,0);
        \draw (1,-1) -- (1,0);
        \draw (-1,-1) -- (1,0);
        \draw (-1,0) -- (-1,1);
        \draw (-1,0) -- (1,1);
        \draw (1,0) -- (1,1);
         \filldraw[black] (-5,0) circle (2pt) node[anchor=north]{$x_4$};
        \filldraw[black] (-3,0) circle (2pt) node[anchor=north]{$y_4$};
        \filldraw[black] (-1,-1) circle (2pt) node[anchor=north]{$x_1$};
        \filldraw[black] (1,-1) circle (2pt) node[anchor=north]{$y_1$};
        \filldraw[black] (-1,0) circle (2pt) node[anchor=east]{$x_2$};
        \filldraw[black] (1,0) circle (2pt) node[anchor=west]{$y_2$};
        \filldraw[black] (-1,1) circle (2pt) node[anchor=south]{$x_3$};
        \filldraw[black] (1,1) circle (2pt) node[anchor=south]{$y_3$};
    \end{tikzpicture}
    \caption{A poset with critical pairs}
    \label{fig:critpairs}
\end{figure}
The pair $(x_3,y_3)$ is critical because the first condition is satisfied and the second is vacuously true. The pair $(x_1,y_1)$ is critical because the first condition is vacuously true and the second condition is satisfied. The critical pair $(x_2,y_2)$ satisfies both conditions, and $(x_4,y_4)$ is trivially a critical pair. Observe that $(y_4,x_4)$ is also a critical pair, but $(y_1,x_1)$ is not. In general this relation is not symmetric, in contrast to incomparability. 
\end{example}
The connections to hypergraphs are a prominent part of what make the objects enumerated in this work notable, so some definitions there will also be needed. 

A \textit{simple graph} $G=(V,E)$ is an ordered pair where $V$ is a nonempty set of \textit{vertices} and $E$ is a nonempty set of 2-element subsets of $V$, the \textit{edges}. This differs from a general graph by its lack of loops (edges of cardinality 1) and its forbidding of multiple edges between the same vertices. Graphs used henceforth will be simple.   
A \textit{vertex coloring} of a graph $G$ is a function $c:V \to \nats$ such that $c(u) =c(v)$ implies $\{u,v\} \not\in E$.
The least $k\in \nats$ such that $G$ has a vertex coloring $c:V \to \nats$ with $|c(V)|=k$ is the \textit{chromatic number} of $G$, denoted $\chi(G)$. 
A \textit{simple hypergraph} $\mathcal{H}=(V,\mathcal{E})$ is a set of vertices $V$ and a set of \textit{hyperedges} $\mathcal{E}$, where each hyperedge $e \in \mathcal{E}$ is a subset of $V$ with at least 2 elements. Again this differs from general hypergraphs by excluding loops (single-element sets) and multiple edges on the same vertices. Hypergraphs henceforth will be simple.  
A \textit{weak vertex coloring} of a hypergraph is a function $c: V \to \nats$ such that for every hyperedge $e \in \mathcal{E}$, there exist vertices $u,v \in e$ such that $c(u) \neq c(v)$.  
When $|e|= 2$, then the notions of vertices sharing an edge are uniquely colored and vertices sharing an edge are not uniformly colored are equivalent. A vertex coloring of a hypergraph where each vertex in a hyperedge is assigned a unique color is called a \textit{strong} hypergraph coloring, as this is a more stringent property. If the hypergraph is a graph, both weak and strong hypergraph vertex coloring coincide with each other and the prior definition for graphs.
The \textit{weak chromatic number} $\chi(\mathcal{H})$ (respectively, \textit{strong chromatic number} $\chi_s(\mathcal{H}))$ of a hypergraph is the least positive integer $k$ such that $\mathcal{H}$ has a weak (strong) vertex coloring $c:V \to \reals$ with $|c(V)|=k$. The strong and weak hypergraph chromatic number also coincide with the graph chromatic number when the hypergraph is a graph. 

We now recall some useful properties of posets, most of which can be found in~\cite{Trotter}.
First, recall that a poset $\Pos$ is the intersection of all its linear extensions $\{L_i : i\in I\}$, 
so $\Pos = \bigcap_{i \in I}  L_i$, where $x \leq y$ in $\Pos$ iff $x \leq y$ in $L_i$ for all $i$. A family of linear extensions whose intersection as above is $\Pos$ is said to \textit{realize} $\Pos$, or be a \textit{realizer} of $\Pos$.  Such a realizer must satisfy a number of properties, but two of interest are in the following theorem:
\begin{theorem}
\label{thm:reverse}
Let $R$ be a family of linear extensions $L_i$ of $\Pos=(X,\leq)$. Then the following are equivalent:
\begin{itemize}
    \item $R$ is a realizer of $\Pos$.
    \item For each pair $(x,y) \in \inc(\Pos)$, $R$ contains both a linear extension $L_i$ with $x<y$ and a linear extension $L_j$ with $y<x$.
    \item For each critical pair $(x,y) \in \crit(\Pos)$, $R$ must contain a linear extension $L_i$ with $y<x$.
\end{itemize}
\end{theorem}
The minimum cardinality of a realizer of $\Pos$ is called the \textit{order dimension} of $\Pos$, denoted $\odim(\Pos)$. 
\begin{observation}
    \label{obs:reals}
    A poset $\Pos=(X,\leq)$ can be embedded into $(\reals^d,\leqslant)$ with its usual partial order, denoted $\Pos \hookrightarrow (\reals^d,\leqslant),$ if and only if $\Pos$ has order dimension at most $d$.
\end{observation}
The following defines our main objects of study in this article. 
\begin{definition}
For a poset ${\Pos} = (X,\leq)$ an \textit{alternating cycle} $S$
is a cyclically ordered tuple
\[
S = [(x_1,y_1),(x_2,y_2),\ldots,(x_k,y_k)]
\]
where each $(x_i,y_i) \in \inc(\Pos)$ 
and $y_i \leq x_{i+1}$ for each $i$ in a cyclic fashion modulo $k$. Two such alternating cycles $S$ and
$S' = [(x_1',y_1'),(x_2',y_2'),\ldots,(x_k',y_k')]$ are considered identical
iff there is a $ c\in \{0,1,\ldots,k-1\}$ such that
$(x_i',y_i') = (x_{i+c},y_{i+c})$ for each $i$ modulo $k$. Let $\mathcal{AC}(\Pos)$ 
denote the set of alternating cycles of $\Pos$. 

An alternating cycle is \textit{strict} if $y_i \leq x_j$ implies that $j=i+1$ modulo $k$. 
Let $\mathcal{AC}^s(\Pos)$ denote the set of strict alternating cycles of $\Pos$.
\end{definition}
\begin{example}
    Consider $\mathbb{Z}$ equipped with $x\leq'y$ iff $x|y$. Then $[(9,2)(4,7)(21,5)(10,3)]$ is an alternating cycle that is not strict (as $2|10$ and $3|21$) and 
$[(9,2)(4,7)(7,13)(169,3)]$
is a strict alternating cycle.
\end{example}
In \cite{TrotterMoore77} alternating cycles were first deployed to help establish planar posets with a maximum or minimum element have dimension at most 3. Alternating cycles have a number of useful properties. One important property is the following:
\begin{theorem}
\label{thm:cyclefree}
    Let $\Pos=(X,\leq)$ be a poset and $S \subseteq \inc(P)$. Then the following are equivalent:
    \begin{itemize}
        \item ${\tr}(\leq\cup S)$ is not a partial order on $X$
        \item $S$ contains an alternating cycle
        \item $S$ contains a strict alternating cycle
    \end{itemize}
\end{theorem}
Theorem~\ref{thm:cyclefree} can be leveraged to show that poset dimension is indeed 
well-defined, after a quick lemma from \cite{szpilrajn1930extension}:
\begin{lemma}
    \label{lemma:extend}
    Let $\Pos =(X, \leq)$ be a poset. Then $\Pos$ has a linear extension.
\end{lemma}
Equipped with this, the following holds:
\begin{theorem}
    \label{thm:welldef}
    Let $\Pos=(X,\leq)$ be a poset and let $\mathcal{E}(\Pos)$ be the set of all linear extensions of $\Pos$. Then $\Pos=\bigcap_{L_i \in \mathcal{E}(\Pos)}L_i$ and thus dimension is well-defined.
\end{theorem}
Another consequence of Theorem \ref{thm:cyclefree} is that if $S \subseteq \crit(\Pos)$, there exists a linear extension reversing all pairs in $S$ iff $S^d = \{(y,x): (x,y) \in S\}$ does not contain any (strict) alternating cycles. 

Define the hypergraph $\mathcal{H}_\Pos = (\mathcal{V,E})$ with vertex set $\mathcal{V}$ consisting of elements of $\inc(\Pos)$ and with hyperedge set $\mathcal{E}$ as the subsets of $\mathcal{V}$ whose duals form alternating cycles. Note that alternating cycles need have at least two incomparable pairs, so this is a simple hypergraph. Analogously, the hypergraph $\mathcal{H}^s_\Pos$ is the same but with strict alternating cycles dual to its edges, as well as $\mathcal{K}_\Pos$ and $\mathcal{K}^s_\Pos$ denoting the same structures but using $\crit(\Pos)$ for the vertex sets. This establishes the following theorem 
from~\cite{Trotter}. Since no explicit proof is provided in~\cite{Trotter} we include one here for self containment. 
\begin{theorem}
    Let $\Pos=(X,\leq)$ be a poset that is not a chain. Then:
    $$\odim(\Pos) = \chi(\mathcal{H}_\Pos)=\chi(\mathcal{H}^s_\Pos)=\chi(\mathcal{K}_\Pos)=\chi(\mathcal{K}^s_\Pos).$$
\end{theorem}
\begin{proof}
    All four cases of $\mathcal{H}(\Pos),\mathcal{H}^s(\Pos),\mathcal{K}(\Pos)$, and $\mathcal{K}^s(\Pos)$ are shown similarly, so we will show that 
    $\chi(\mathcal{H}^s(\Pos))=\odim(\Pos)$.
    
    Suppose $\odim(\Pos)=d$, so there exist linear extensions $L_1,\dots,L_d$ realizing $\Pos$ as $\Pos=\bigcap_{i=1}^d L_i$. Clearly 
    no $L_i$ reverses all pairs of the dual of a strict alternating cycle. Hence, we can assign color 1 to each vertex (incomparable pair $(x,y)\in \inc(\Pos)$) reversed in $L_1$, color 2 to each reversed by $L_2$, etc. This process halts after at most
    $d$ steps and it yields a (weak) $d$-coloring of $\mathcal{H}^s(\Pos)$, and so $\chi(\mathcal{H}^s(\Pos))\leq d$.
    
    Conversely, if $\chi(\mathcal{H}^s(\Pos))=d$, then there is no monochromatic hyperedge in $\mathcal{H}^s(\Pos)$. By Theorem \ref{thm:reverse} and Lemma \ref{lemma:extend}, there exists a linear extension $L_i$ of ${\tr}(\leq\cup S_i)$, where $S_i$ reverses each pair assigned the color $i$. Then the $L_i$ form a family of linear extensions of cardinality at most $d$ that reverse each pair $(x,y)\in \inc(\Pos)$ at least once, and thus $\odim(\Pos)\leq d=\chi(\mathcal{H}^s(\Pos)).$ 
\end{proof}
With these parallels to (hyper)graph coloring established, it should be no surprise that computing poset order dimension is NP-complete \cite{YannakakisNP}. The order dimension of small posets can be computed, and finding upper bounds can be made tractable, as in \cite{yanez1999poset}, but the general problem quickly becomes infeasible. Having established the appropriate definitions, we can now enumerate strict alternating cycles of a general poset.

\section{Enumeration for the Trivial Poset $\Pos_0$}

Sometimes, starting on this type of problem requires one to first simply establish one case that can be counted, and so one proceeds by considering the simplest case first. Other times it may be most useful to ponder what the limiting case may be, and then enumerate that one. In this case, these coincide: the trivial poset $\Pos_0$ (see Section~\ref{sec:intro} or Definition~\ref{def:trivial} here below) is both perhaps the simplest to consider and ends up being the maximal case. To make things abundantly clear we make the following definition.
\begin{definition}
\label{def:trivial}
    Let the \textit{trivial poset} $\Pos_0=(X,=)$ denote the poset whose only relations are of the form $x\leq y \iff x=y$ for all $x,y\in X$, and therefore $(x,y) \in \inc(\Pos) \iff x\neq y$. 
\end{definition}
For $\Pos_0$ all links in alternating cycles will happen via $y_i=x_{i+1}$. Because of this, a strict alternating cycle can represented as a directed cycle where each vertex is an element of $X$.
Since any $n$-element subset of $X$ with $n\geq 2$, given a cyclic order, can be a strict alternating cycle in this way, the number of strict alternating cycles is equal to the number of unique cyclic orderings of such $n$-element subsets of $X$. Figure \ref{fig:digraphcycle} depicts one of $5!=120$ possible on its 6 mutually incomparable elements. 
\begin{figure}[h!]
    \centering
    \begin{tikzpicture}
        \node (A) at (-2, 0) {$x_1$};
        \node (B) at (-1, -2) {$x_2$};
        \node (C) at (1, -2) {$x_3$};
        \node (D) at (2, 0) {$x_4$};
        \node (E) at (1, 2) {$x_5$};
        \node (F) at (-1, 2) {$x_6$};
        \draw [->] (A) edge (B) (B) edge (C) (C) edge (D) (D) edge (E) (E) edge (F) (F) edge (A);
    \end{tikzpicture}
    \caption{Alternating cycle $S =\left[(x_1,x_2)(x_2,x_3)(x_3,x_4)(x_4,x_5)(x_5,x_6)(x_6,x_1)\right]$ as a digraph cycle}
    \label{fig:digraphcycle}
\end{figure}
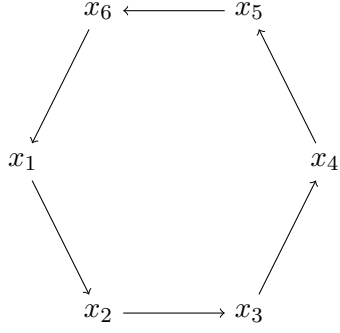

In this way we obtain the following proposition:

\begin{proposition}
    \label{prp:trivposet}
    Let $\Pos_0=(X,=)$ be the trivial poset on $|X|=n$ elements. Then the number of strict alternating cycles on $\Pos_0$ is given by: 
    $$|\mathcal{AC}^s(\Pos_0)|=\sum_{k=2}^n \binom{n}{k}(k-1)!=n!g(n)$$
    where
    $$g(n) = \sum_{k=0}^{n-2}\frac{1}{k!}\frac{1}{n-k}.$$
\end{proposition}
From Proposition~\ref{prp:trivposet} it is clear the quantity $|\mathcal{AC}^s(\Pos_0)|$ is ultimately at least as large as the $(n-1)!$ term appearing in the $k=n$ term, as will be explored shortly, and therefore grows extremely rapidly. A table of values is included in the Appendix in Table \ref{table:1}. 

Using  $g(n)$ has a number of advantages. One is that it becomes very easy to talk about the associated exponential generating function. Given a full (and at times humorous) treatment in \cite{Graham-Knuth-Patashnik}, the \textit{exponential generating function} of a sequence $(g(n))_{n\geq 0}$ is a function given by 
$$\hat{G}(x)=\sum_{n \geq 0} g(n) \frac{x^n}{n!}.$$ 
This is often useful, as it allows analytic tools to be leveraged at otherwise discrete problems. 

Consider the exponential generating function $c_{{\Pos}_0}(x)$ of the
number of alternating cycles of ${\Pos}_0$ in terms of $n = |X|$. If $g(n)$  is as defined in Proposition \ref{prp:trivposet}, the corresponding sequence $(g(n))_{n\geq 0}$ is visibly the discrete convolution of the sequences $(f(n))_{n \geq 0}$ and $(h(n))_{n \geq 0}$ where $f(n)=\frac{1}{n!}$ and $h(n)= \frac{1}{n}$ for $n\geq 2$ and $h(0)=h(1)=0$. Then the generating function $g(x)$ is given by the Cauchy product of the respective generating functions $f(x)=\sum_{n=0}^\infty\frac{x^n}{n!}=e^x$ and $h(x)=\sum_{n=2}^{\infty}\frac{x^n}{n}=-\ln(1-x)-x$. Thus, the following holds for $c_{\Pos_0}(x)$:
\begin{observation}
    The exponential generating function $c_{\Pos_0}(x)$ is given by:
    $$c_{\Pos_0}(x)=\sum_{n=0}^\infty (n! g(n))\frac{x^n}{n!}=\sum_{n=0}^\infty g(n)x^n=-e^x(\ln(1-x)-x).$$
\end{observation}
This is noteworthy for simply having a nice closed form. It bears interesting similarity to a formula arising from the ``samplesort'' algorithm (see page 354 of \cite{Graham-Knuth-Patashnik}), which is a divide-and-conquer sorting algorithm generalizing quicksort \cite{frazer1970samplesort}. That a key step in optimal performance of this algorithm is dividing into approximately equal size ``buckets'' is probably not a coincidence, given results appearing in the next section on posets of bounded width.

\section{$\Pos_0$ as the Upper Bound}

Having established the number of strict alternating cycles possible on $\Pos_0=(X,=)$ when $|X|=n$, it is natural to next consider where that fits in relative to some general poset on the same ground set, $\Pos=(X,\leq)$. What comparisons can be made between two posets with the same ground set but differing relations? The easiest case to consider is when one poset has all the relations of the other, plus some additional relations. Formally,  consider
$\Pos' = (X,\leq')$ such that $\Pos=(X,\leq)$ is an extension of $\Pos'$. Recall that the identity map $X\rightarrow X$ induces an order preserving 
embedding $\Pos'\hookrightarrow\Pos$ and $x\leq'y$ in $\Pos'$ implies $x\leq y$
in $\Pos$.  Now consider alternating cycle $S = [(x_1,y_1),(x_2,y_2),\ldots,(x_k,y_k)]$ in $\Pos$. Because $x_i$ and $y_i$ are incomparable in 
$\Pos$ for each $i$, it must be that $x_i\neq y_i$ for each $i$ and that 
$x_i$ and $y_i$ are incomparable in $\Pos'$. Then an alternating cycle $S'$ in $\Pos'$ can be created from $S$ as we will now demonstrate: 
because $y_i \leq x_{i+1}$  in $\Pos$ for each $i$ in a cyclic fashion modulo $k$, either $y_i\leq' x_{i+1}$ or $y_i||x_{i+1}$ in $\Pos'$. If $y_i||x_{i+1}$ in ${\Pos}',$ add the ordered tuple $(y_i,x_{i+1})$ into $S$
between $(x_i,y_i)$ and $(x_{i+1},y_{i+1})$. Doing this for each $i$, cyclically modulo 
$k$, we obtain an alternating cycle $S'$ in the poset $\Pos'$ that has ``fixed'' any links that are valid in $\Pos$ but not $\Pos'$. Using this construction, one can map alternating cycles from any $\Pos$ to a poset $\Pos'$ on the same groundset with a subset of its relations. In fact, we will see that this preserves strictness, but that will require proving. Notably, the poset with the smallest possible subset of relations for a given groundset $X$ will be $\Pos_0=(X,=)$. First, consider a quick example. 
\begin{example}
\label{exa:S-S'}
Suppose $S= [(x_1,y_1),(x_2,y_2),(x_3,y_3),(x_4,y_4),(x_5,y_5)]$
is a strict alternating cycle in a poset ${\Pos}$ where
$y_1 < x_2,~y_2 = x_3,~y_3 < x_4,~ y_4 <~x_5,~y_5 = x_1$. In this case
the corresponding alternating cycle $S'$ in the trivial poset 
${\Pos}_0 = (X,=)$ is given by
\[
S' =
[(x_1,y_1),(y_1,x_2),(x_2,y_2),(x_3,y_3),
(y_3,x_4),(x_4,y_4),(y_4,x_5),(x_5,y_5)]. 
\]
\end{example}
For a given image alternating cycle $S'$, the alternating cycle  $S$ can be retrieved by removing each ordered tuple of comparable elements in $\Pos$ 
from $S'$. Therefore, we have the following:
\begin{observation}
\label{obs:nonconsec-inj}
Let $\Pos = (X,\leq)$ and $\Pos' = (X,\leq')$ be posets with $\Pos' \subseteq \Pos$ as above where 
the identity map on $X$ yields an order preservation $\Pos'\rightarrow\Pos$.
The map ${\cal{AC}}({\Pos})\rightarrow {\cal{AC}}({\Pos'})$ given by
$S\mapsto S'$ is well-defined and injective. By definition, each added 
tuple in $S'$ is always between two ordered tuples originally in $S$, and therefore no two consecutive tuples will have been added.
\end{observation}
As mentioned, it is time to establish that this map preserves the property of strictness: 
if $S$ is strict in $\Pos$, we will show that $S'$ is strict in $\Pos'.$ 
\begin{proof}
Consider an alternating cycle 
$S = [(x_1,y_1),(x_2,y_2),\ldots,(x_k,y_k)]$ in $\Pos$, and corresponding 
alternating cycle $S'$ in $\Pos'$ denoted 
$S' = [(x_1',y_1'),(x_2',y_2'),\ldots,(x_h',y_h')],$ where $h\geq k$.
Suppose that $S'$ is not a strict alternating cycle. Because $S'$ is cyclic, it can be assumed that $y_1'\leq'x_i'$, for some  $i\in\{3,\ldots,h\}$. Then there are four cases to consider: 

{\sc First case:} Neither the tuple $(x_1',y_1')$ nor the tuple
$(x_i',y_i')$ has been added to $S$ obtain $S'$, and so both are in $S$. 
If $i=3$ and the tuple $(x_2',y_2')$ was added to $S$ to obtain $S'$, then by assumption we have $x_2' = y_1' \leq' x_3'= y_2'$ in $\Pos'$ 
contradicting that $x_2'$ and $y_2'$ are incomparable in $\Pos'$. Hence, by Observation~\ref{obs:nonconsec-inj}, either $i\geq 4$ 
or $(x_2',y_2')$ was in $S$. In either case there is at least one non-added tuple from $S$ between $(x_1',y_1')$ and $(x_i',y_i')$. 
Therefore $S$ is also not strict.

{\sc Second case:} The tuple $(x_1',y_1')$ has been added to $S$ and
$(x_i',y_i')$ has not been added to $S$ to obtain $S'$. By Observation~\ref{obs:nonconsec-inj}
both the tuples $(x_h',y_h')$ and $(x_2',y_2')$ are from $S$ and were not added. 
In this case, since $(x_1',y_1')$ was added, we have $x_1' = y_h' < x_2' = y_1'$ in $\Pos$. 
Also, since $y_1'\leq'x_i'$ we have $y_h' = x_1' < y_1'\leq x_i'$ in $\Pos$. Since $i\neq 2$, $S$ must not be strict. 

{\sc Third case:} The tuple $(x_1',y_1')$ has not been added to $S$ and
$(x_i',y_i')$ has been added to $S$ to obtain $S'$. By Observation~\ref{obs:nonconsec-inj}
both the tuples $(x_{i-1}',y_{i-1}')$ and $(x_{i+1}',y_{i+1}')$ are 
from $S$ and were not added. In this case $y_1' \leq' x_i' < y_i'\leq x_{i+1}'$ 
in $\Pos$ and there
is at least one non-added tuple (namely, the tuple $(x_{i-1}',y_{i-1}')$) between $(x_1',y_1')$
and $(x_{i+1}',y_{i+1}')$. If $i=h$, and so cyclically $i+1 = 1$, then $y_1'\leq x_1'$
contradicting $x_1'||y_1'$. Therefore $i\leq h-1$ and hence $S$ is not strict.

{\sc Fourth case:} Both the tuples $(x_1',y_1')$ and $(x_i',y_i')$ were  added to 
$S$ to obtain $S'$. By Observation~\ref{obs:nonconsec-inj} we must have $3\leq i\leq h-1$ and
neither $(x_h',y_h')$ nor $(x_{i+1}',y_{i+1}')$ were added. In this case
$y_h' = x_1' < y_1'\leq x_i' < y_i' = x_{i+1}'$ in $\Pos$. Since $x_h'||y_h'$ in 
$\Pos,$ it therefore holds that $i+1\neq h$ and so $i\leq h-2$. Thus $S$ is not strict. 

\end{proof}
Therefore the following proposition holds:
\begin{proposition}
\label{prp:stict2strict}
Let $\Pos = (X,\leq)$ and $\Pos' = (X,\leq')$ be posets  with $\Pos' \subseteq \Pos$ as above, where 
the identity map on $X$ yields an order preservation $\Pos'\rightarrow\Pos$.
The map $S\mapsto S'$ from Observation~\ref{obs:nonconsec-inj} yields an 
injection ${\cal{AC}}^s({\Pos})\hookrightarrow {\cal{AC}}^s({\Pos'})$.
In particular,  $|{\cal{AC}}^s({\Pos})|\leq |{\cal{AC}}^s({\Pos'})|$.
\end{proposition}
For any given poset, removing relations therefore increases the number of strict alternating cycles. For a given groundset $X$, removing as many relations as possible from $\Pos=(X,\leq)$ yields the trivial poset $\Pos_0 = (X,=)$. Therefore:
\begin{proposition}
\label{prp:max-trivial}
The maximum number 
${\cal{MAC}}^s(n) = \max\{|{\cal{AC}}^s({\Pos})| : |X| \leq n\}$ 
of strict alternating cycles of a poset $\Pos = (X,\leq)$ with $|X|\leq n$ 
satisfies ${\cal{MAC}}^s(n) = |{\cal{AC}}^s({\Pos}_0)|$ 
, where ${\Pos}_0 = (X,=)$ is the trivial poset with $|X|=n$. 
\end{proposition}
One quirk worth addressing is that this means that the associated hypergraphs of $\Pos_0$ have the most possible hyperedges, yet have chromatic number of only 2, as that is the order dimension of this very simple poset. Essentially, the small chromatic number is no obstacle to having many hyperedges simply because the associated hypergraphs to $\Pos_0$ have so many more vertices, as every pair of differing elements is both incomparable and critical.  

\section{Asymptotics of $|\mathcal{AC}^s(\Pos_0)|$ and $\mathcal{MAC}^s(\Pos)$}

It has been shown that the poset $\Pos_0=(X,=)$ has the maximum alternating cycles possible for posets on a ground set of size $|X|\leq n$. It bears examining how this quantity grows as $n$ does. Inspection of  Proposition \ref{prp:trivposet} will quickly reveal that the factor of $(n-1)!$ arising from the $k=n$ term in the sum will be dominant as $n$ grows toward $\infty$, but more can be said. Once again writing $|\mathcal{AC}^s\Pos_0|$ as $n!g(n)=n!\sum_{k=0}^{n-2} \frac{1}{k!}\frac{1}{n-k}$, asymptotics of $g(n)$ can be established as follows:
\begin{equation}
g(n) = 
\sum_{k=0}^{n-2}\frac{1}{k!}\frac{1}{n-k} = 
\frac{1}{n}\sum_{k=0}^{n-2}\frac{1}{k!}\left(1 + \frac{k/n}{1-k/n}\right) = 
\frac{1}{n}\left[\sum_{k=0}^{n-2}\frac{1}{k!} +
  \frac{1}{n}\sum_{k=1}^{n-2}\frac{1}{(k-1)!}\frac{n}{n-k}\right].
\end{equation}
Making use of the fact that $e = \sum_{k=0}^{\infty}\frac{1}{k!}$ we can obtain:
\begin{equation}
\label{eqn:2sum}
g(n) = \frac{1}{n}\left[\left(e - \sum_{k=n-1}^{\infty}\frac{1}{k!}\right)
  + \frac{1}{n}\sum_{k=1}^{n-2}\frac{1}{(k-1)!}\frac{n}{n-k}\right]
= \frac{1}{n}\left[\left(e - \Sigma_1(n)\right) + \frac{1}{n}\Sigma_2(n)\right].
\end{equation}
Here $\Sigma_1(n)$ and $\Sigma_2(n)$ are given by:
$$\Sigma_1(n) = \sum_{k=n-1}^{\infty}\frac{1}{k!}~~ \text{and}~~
\Sigma_2(n) = \sum_{k=1}^{n-2}\frac{1}{(k-1)!}\frac{n}{n-k}.$$
Now, for $a,b$ positive integers with $0<a\leq b$, it holds that
\begin{equation}
\label{eqn:e-fact}
\sum_{k=a}^b\frac{1}{k!} < \frac{1}{a!}\sum_{k=0}^{\infty}\frac{1}{(a+1)^k}
= \frac{1}{a!}\frac{a+1}{a}< \frac{2}{a!}, 
\end{equation}
and therefore $\Sigma_1(n)$ can be bounded from above and below by 
\begin{equation}
\label{eqn:sum1}
\frac{1}{(n-1)!} < \Sigma_1(n) < \frac{2}{(n-1)!}.
\end{equation}
Next, by splitting $\Sigma_2(n)$ of (\ref{eqn:2sum})
at midpoint $m = \lfloor n/2\rfloor$ into two sums, 
by noting that $\frac{n}{n-k} > 1$ for $k\leq m$, it holds that:
\begin{eqnarray*}
\Sigma_2(n) & = & \sum_{k=1}^{n-2}\frac{1}{(k-1)!}\frac{n}{n-k} \\
& = & \sum_{k=1}^{m}\frac{1}{(k-1)!}\frac{n}{n-k} +
\sum_{k=m+1}^{n-2}\frac{1}{(k-1)!}\frac{n}{n-k} \\
  & > & \sum_{k=1}^m\frac{1}{(k-1)!}  + \frac{1}{m!} \\
  & > & 2 + \frac{1}{(\lfloor n/2\rfloor)!}
\end{eqnarray*}
as long as $m\geq 2$ and therefore $n\geq 4$. 
Subsequently, by noting that $n/(n-k) < 2$
when $k\in\{1,\ldots,m\}$ 
and $n/(n-k) < n/2$ for $k\in\{m+1,\ldots,n-2\}$ and utilizing
(\ref{eqn:e-fact}), it follows that 
\begin{eqnarray*}
\Sigma_2(n) & = & \sum_{k=1}^{n-2}\frac{1}{(k-1)!}\frac{n}{n-k} \nonumber \\
& = & \sum_{k=1}^{m}\frac{1}{(k-1)!}\frac{n}{n-k} +
\sum_{k=m+1}^{n-2}\frac{1}{(k-1)!}\frac{n}{n-k} \nonumber \\
  & < & 2e + \frac{n}{2}\sum_{k=m+1}^{\infty}\frac{1}{(k-1)!} \nonumber \\
  & < & 2e + \frac{n}{\lfloor n/2\rfloor!}.
\end{eqnarray*}
Combining these two inequalities on $\Sigma_2(n)$ yields the bounds:
\begin{equation}
\label{eqn:sum2}
2 + \frac{1}{\lfloor n/2\rfloor!} < \Sigma_2(n) <
2e + \frac{n}{\lfloor n/2\rfloor!}.
\end{equation}
By (\ref{eqn:2sum}), (\ref{eqn:sum1}), and (\ref{eqn:sum2}),
the following concrete bounds for 
$g(n)$  are obtained:
\begin{equation}
\label{eqn:gn-concr}  
\frac{e}{n} + \frac{2}{n^2} + \frac{1}{n^2\lfloor n/2\rfloor!} - \frac{2}{n!} 
< g(n) < 
\frac{e}{n} + \frac{2e}{n^2} + \frac{1}{n\lfloor n/2\rfloor!} - \frac{1}{n!}.
\end{equation}
Observe that as $n \to \infty$, the dominant term of both upper and lower bounds is $\frac{e}{n}$, and therefore asymptotically $g(n) \sim \frac{e}{n}$. Thus
making use of Proposition \ref{prp:trivposet}, Proposition \ref{prp:max-trivial}, and (\ref{eqn:gn-concr}) gives the following:
\begin{proposition}
\label{prp:alt-cycles-trivial}
The trivial poset ${\Pos}_0 = (X,=),$ on $|X| = n$ vertices, satisfies
$$
|{\cal{AC}}^s({\Pos_0})| = \sum_{k=2}^n\binom{n}{k}(k-1)!=n!g(n) \sim e(n-1)!$$
as $n$ tends toward infinity. Therefore, for large $n$, the maximum number of alternating cycles of any poset on $X$ with $|X|\leq n$ satisfies:
$${\cal{MAC}}^s(n) = |{\cal{AC}}^s({\Pos}_0)|\sim e(n-1)!~.$$
\end{proposition}
This is not quite as elegant as one might hope: in a perfect world it might be that $\mathcal{MAC}^s(n) \leq e(n-1)!$ would hold, but that is not the case. In this reality $e(n-1)!$ is not an upper bound, but rather a tight lower bound on the upper bound. Some values of $e(n-1)!$ in comparison to $\mathcal{MAC}^s(n)$  are included in Table \ref{table:1}. As that table illustrates (and is clear from the equations), the number of possible strict alternating cycles on a ground set of order $n$ is enormous, but at least the bounding case of the trivial poset $\Pos_0$ is an uncommon case. Posets arising from scenarios of interest tend to have many more relations, even if they remain sparse, and in non-trivial cases we will next see that a tighter bound can be obtained.

\section{Dilworth's Theorem}

While understanding the bounding case of $\Pos_0$ is useful, posets of interest tend to have many, many more relations than simply the reflexive ones. It is natural, then, to want to have bounds that can be adapted to more parameters than simply the cardinality of the ground set. The titular theorem for this section, originally from \cite{Dilworth}, will be a useful tool. 

Some conventions: for $\Pos = (X,\leq)$, we say that $\Pos$ is \textit{decomposed into chains} or is a \textit{union of chains} if there exists a set partition $X=X_1\cup \dots \cup X_n$ such that each $\chain_i=\Pos[X_i]$ is a chain.
Here $\chain_i$ will be used to denote chains as posets, where $X_i$ will denote underlying sets of chains, which may in the poset have relations outside themselves. This distinction is made because:
\begin{theorem}[Dilworth's Theorem]
    Any finite poset $\Pos=(X,\leq)$ can be written as  a union of $w$ disjoint chains, where $w$ is the width of $\Pos$.
\end{theorem}
Note that multiple such coverings may be possible for a given poset, assuming $w \geq 2$. Alternating cycles fit nicely into these decompositions
for disjoint union of chains: to make a jump from $(x_1,y_1)$ to $(x_2,y_2)$, it must be the case that $y_1\leq x_2$, and therefore there must be a chain containing both $y_1$ and $x_2$. 
\begin{example}On the width-4 poset $\chain$ shown in Figure \ref{fig:chains}, one possible alternating cycle would be 
$[(a_1,b_2)(b_3,c_2)(c_2,d_2)(d_4,a_1)]$. 
\end{example} 
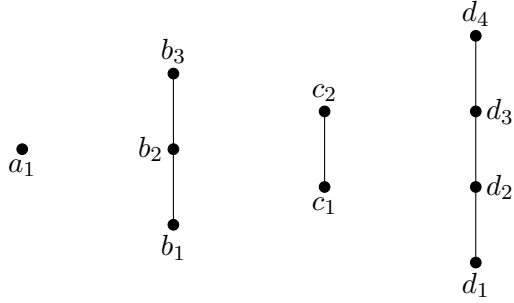
\begin{figure}[h!]
    \centering
    \begin{tikzpicture}
        \draw (-3,-1) -- (-3,0);
        \draw (-3,0) -- (-3,1);
        \draw (-1,-0.5) -- (-1,0.5);
        \draw (1,-1.5) -- (1,-0.5);
        \draw (1,-0.5) -- (1,0.5);
        \draw (1,0.5) -- (1,1.5);
         \filldraw[black] (-5,0) circle (2pt) node[anchor=north]{$a_1$};
        \filldraw[black] (-3,-1) circle (2pt) node[anchor=north]{$b_1$};
        \filldraw[black] (-3,0) circle (2pt) node[anchor=east]{$b_2$};
        \filldraw[black] (-3,1) circle (2pt) node[anchor=south]{$b_3$};
        \filldraw[black] (-1,-0.5) circle (2pt) node[anchor=north]{$c_1$};
        \filldraw[black] (-1,.5) circle (2pt) node[anchor=south]{$c_2$};
        \filldraw[black] (1,-1.5) circle (2pt) node[anchor=north]{$d_1$};
        \filldraw[black] (1,-0.5) circle (2pt) node[anchor=west]{$d_2$};
        \filldraw[black] (1,0.5) circle (2pt) node[anchor=west]{$d_3$};
        \filldraw[black] (1,1.5) circle (2pt) node[anchor=south]{$d_4$};
    \end{tikzpicture}
    \caption{A poset of four disjoint chains: $\Pos=\chain_1\cup \chain_2 \cup \chain_3 \cup \chain_4$}
    \label{fig:chains}
\end{figure}
Enumerating all the possible strict alternating cycles will require some care notationally. For a ground set $X$ with $|X| = n$ and a positive integer $w\in\nats$, let 
$\tilde{n} = (n_1,\ldots,n_w)$ be an ordered $w$-tuple of positive integers
that add up to $n$,
so $\sum_i n_i = n$. For a partition $X = X_1\cup\cdots\cup X_w$ with
$|X_i| = n_i$ for each $i$,
denote the poset consisting of $w$ disjoint chains $\chain_1,\ldots,\chain_w$ on
the sets $X_1,\ldots,X_w$ respectively by ${\chain}(\tilde{n})$. Assume a fixed labeling of the elements
in $X$ and a fixed total
order in each of
the chains $\chain_i$. Thus we write
${\chain}(\tilde{n}) = \chain_1\cup\cdots\cup \chain_w$ as a
poset. Let $S = [(x_1,y_1),(x_2,y_2),\ldots,(x_k,y_k)]$ 
be a strict alternating cycle in 
${\chain}(\tilde{n})$. In this case $y_i\leq x_{i+1}$ for each $i$,
cyclically modulo $k$, and hence
the pair $\{y_i,x_{i+1}\}$ is contained in the same chain $\chain_{i+1}$ among
$\chain_1,\ldots,\chain_k$ (cyclically) for each $i$.  
Suppose a chain $\chain_i$ contains two $x$-elements of the alternating
cycle $S$, $x_a$ and $x_b$.
If $x_a\leq x_b$, then $y_{b-1} \leq x_b \leq x_a$ and hence $S$ is not
strict. A similar conclusion
is reached if $x_b\leq x_a$ and so a chain $\chain_i$ cannot contain two
$x$-elements from $S$.
In the same way, if a chain $\chain_i$ contains two $y$-elements of $S$,  $y_a$ and
$y_b$, then if $y_a\leq y_b$
we have $y_a\leq y_b \leq x_{b+1}$ and hence $S$ would not be strict. The same argument holds if $y_b \leq y_a$. Therefore a chain $\chain_i$ of ${\chain}(\tilde{n})$ contains at
most one unordered pair 
$\{y_i,x_{i+1}\}$ of elements from the strict alternating cycle $S$. 
Since the
chain $\chain_i$ has $n_i\geq 1$
elements, there are $\binom{n_i + 1}{2}$ ordered pairs $(y,x)$ with
$x,y\in X_i$ and $y\leq x$. 
Since $S$ is strict, then $k\geq 2$ and the following
observation holds:
\begin{observation}
\label{obs:no-of-sac}
The number $|{\cal{AC}}^s(\chain(\tilde{n})|$ of strict alternating cycles
of the 
poset ${\chain}(\tilde{n}) = \chain_1\cup\cdots\cup \chain_w$ is given by
\[
|{\cal{AC}}^s({\chain}(\tilde{n}))| = 
\sum_{k = 2}^w\left(\sum_{I\in \binom{[w]}{k}}(k-1)!\prod_{i\in I}
\binom{n_i+1}{2}\right) =
\sum_{k = 2}^w(k-1)!\left(\sum_{I\in \binom{[w]}{k}}\prod_{i\in I}
\binom{n_i+1}{2}\right),
\]
where $[w] = \{1,\ldots,w\}$ and $\binom{[w]}{k}$ is the set of all
$k$-element subsets of $[w]$.
\end{observation}
Note that any poset $\Pos=(X,\leq)$ where $X=X_i \cup \dots \cup X_w$ that $\chain(\tilde{n})$ is covering according to Dilworth's Theorem will have fewer alternating cycles, as a consequence of Proposition \ref{prp:stict2strict}. However, this sum can be somewhat cumbersome. Mathematical intuition suggests that the upper bounding case will have uniform (or as close as possible) chain lengths: the next section is devoted to rigorously verifying that intuition.

\section{Tur\'{a}n Posets as the Upper Bound}

This is an optimization problem, and there are many tools available in this domain. That said, the nature of the function in Observation \ref{obs:no-of-sac} of $\tilde{n}=(n_1,\dots,n_w) \in \ints^n$ is ill-suited to techniques such as Lagrange multipliers or using Jensen's-like inequalities. 
What bore fruit was an iterative approach, but showing the desired result requires a somewhat sizable amount of descriptive but admittedly inelegant notation, which we now introduce.

Let $\tilde{x} = (x_1,x_2,x_3,\ldots)$ be a sequence of real variables and 
let $\tilde{a} = (a_0,a_1,a_2\ldots)$ be an increasing sequence of nonnegative real
numbers. For a nonnegative integer $w$ let
\begin{equation}
  \label{eqn:Sw}
S_{w;\tilde{a}}(\tilde{x}) = S_{w;a_0,\ldots,a_w}(x_1,\ldots,x_w) := 
\sum_{k = 0}^w a_k\left(\sum_{I\in \binom{[w]}{k}}\prod_{i\in I}
\binom{x_i+1}{2}\right).
\end{equation}
Note that here $\binom{x_i+1}{2}=\frac{x_i(x_i+1)}{2}$ for any real $x_i$.
By Observation \ref{obs:no-of-sac},  
$|{\cal{AC}}^s({\chain}(\tilde{n}))| = S_{w;\tilde{a}}(\tilde{n})$,
where $a_0 = a_1 = 0$ and $a_k = (k-1)!$ for $k\in\{2,\ldots,w\}$
and $\tilde{n} = (n_1,\ldots,n_w)$. Define the following shifted sums:
\begin{equation}
  \label{eqn:shift}
S_{w;\tilde{a}+}(\tilde{x})  := S_{w;a_1,\ldots,a_{w+1}}(x_1,\ldots,x_w), \ \ 
S_{w;\tilde{a}++}(\tilde{x}) := S_{w;a_2,\ldots,a_{w+2}}(x_1,\ldots,x_w).
\end{equation}
As before, $\binom{A}{k}$ denotes the set of all $k$-element subsets of the set $A$. 
For a set $B$ disjoint from $A$ let
$\binom{A}{k}\uplus B$ denote the set $\{C\cup B : C\in\binom{A}{k}\}$. 
Note that there is a partition of $\binom{[w]}{k}$ as
\begin{equation}
  \label{eqn:part3}
  \binom{[w-2]}{k}
  \cup \left[\binom{[w-2]}{k-1}\uplus \{w-1\}\right]
  \cup \left[\binom{[w-2]}{k-1}\uplus \{w\}\right]
  \cup \left[\binom{[w-2]}{k-2}\uplus \{w-1,w\}\right].
\end{equation}
Consider the $w$-tuple $(x_1,\ldots,x_w)$ and fix the variables
$x_1,\ldots,x_{w-2}$ but vary the last two $x_{w-1}$ and $x_w$, which for
convenience will be renamed (for now) $x$ and $y$ respectively. By (\ref{eqn:part3})
 rewrite $S_{w;\tilde{a}}(\tilde{x})$ from (\ref{eqn:Sw}) as
follows:
\begin{eqnarray*}
  S_{w;\tilde{a}}(\tilde{x})
  & = & \sum_{k = 0}^w a_k\left(\sum_{I\in \binom{[w]}{k}}\prod_{i\in I} 
  \binom{x_i+1}{2}\right) \\
  & = & \lefteqn{ \sum_{k = 0}^w a_k\left(
  \sum_{I\in \binom{[w-2]}{k}}\prod_{i\in I}\binom{x_i+1}{2} +
  \sum_{I\in \binom{[w-2]}{k-1}\uplus\{w-1\}}\prod_{i\in I}\binom{x_i+1}{2}
  \right. } \\
  &   & \left. + \sum_{I\in \binom{[w-2]}{k-1}\uplus\{w\}}\prod_{i\in I}
  \binom{x_i+1}{2} +
  \sum_{I\in \binom{[w-2]}{k-2}\uplus\{w-1,w\}}\prod_{i\in I}\binom{x_i+1}{2} 
  \right).
\end{eqnarray*}
By further factoring out the binomial
expressions in term of $x_{w-1} = x$ and $x_w = y$, it follows from the
above that
\begin{eqnarray*}
  S_{w;\tilde{a}}(\tilde{x})
  & = & \sum_{k = 0}^w a_k\left(
  \sum_{I\in \binom{[w-2]}{k}}\prod_{i\in I}\binom{x_i+1}{2}\right) \\
  &   & + \left(\binom{x+1}{2} + \binom{y+1}{2}\right)\sum_{k = 0}^w a_k\left(
  \sum_{I\in \binom{[w-2]}{k-1}}\prod_{i\in I}\binom{x_i+1}{2}\right) \\
  &   & + \binom{x+1}{2}\binom{y+1}{2}\sum_{k = 0}^w a_k\left(
  \sum_{I\in \binom{[w-2]}{k-2}}\prod_{i\in I}\binom{x_i+1}{2}\right).
\end{eqnarray*}
Since $\binom{[k-2]}{\ell} = 0$ whenever $\ell < 0$ or $\ell > w-2$,
 we obtain from the above that
\begin{eqnarray*}
  S_{w;\tilde{a}}(\tilde{x})
  & = & \sum_{k = 0}^{w-2} a_k\left(
  \sum_{I\in \binom{[w-2]}{k}}\prod_{i\in I}\binom{x_i+1}{2}\right) \\
  &   & + \left(\binom{x+1}{2} + \binom{y+1}{2}\right)
  \sum_{k = 0}^{w-2} a_{k+1}\left(
  \sum_{I\in \binom{[w-2]}{k}}\prod_{i\in I}\binom{x_i+1}{2}\right) \\
  &   & + \binom{x+1}{2}\binom{y+1}{2}\sum_{k = 0}^{w-2} a_{k+2}\left(
  \sum_{I\in \binom{[w-2]}{k}}\prod_{i\in I}\binom{x_i+1}{2}\right) \\
  & = & S_{w-2;\tilde{a}}(\tilde{x})
  + \left(\binom{x+1}{2} + \binom{y+1}{2}\right)S_{w-2;\tilde{a}+}(\tilde{x}) \\
  & & + \left(\binom{x+1}{2}\binom{y+1}{2}\right)S_{w-2;\tilde{a}++}(\tilde{x}).
\end{eqnarray*}
Since $a_i\geq a_{i-1}$ and $x_i\geq 1$ for each $i$, then  clearly 
\[
S_{w-2;\tilde{a}}(\tilde{x})  \leq
S_{w-2;\tilde{a}+}(\tilde{x}) \leq
S_{w-2;\tilde{a}++}(\tilde{x}).
\]
Since $S_{w;\tilde{a}}(\tilde{x})$ is a symmetric polynomial in $x$ and $y$
for fixed $x_1,\ldots,x_{w-2}$, it can be written  as a quadratic
polynomial in terms of $C = x+y$ and $xy$:
\begin{eqnarray}
  S_{w;\tilde{a}}(\tilde{x})
    & = & \frac{S_{w-2;\tilde{a}++}(\tilde{x})}{4}(xy)^2 +
    \left(\frac{S_{w-2;\tilde{a}++}(\tilde{x})}{4}(C+1)
    - S_{w-2;\tilde{a}+}(\tilde{x})\right)xy \nonumber \\
    &   & + S_{w-2;\tilde{a}}(\tilde{x})
    + \frac{S_{w-2;\tilde{a}+}(\tilde{x})}{2}C(C+1) \nonumber \\
    & = & \frac{S_{w-2;\tilde{a}++}(\tilde{x})}{4}\left(xy + \frac{C+1 -
      4S_{w-2;\tilde{a}+}(\tilde{x})/S_{w-2;\tilde{a}++}(\tilde{x})}{2}\right)^2
    \nonumber \\
    &   & + \gamma(S_{w-2;\tilde{a}}(\tilde{x}),S_{w-2;\tilde{a}+}(\tilde{x}),
    S_{w-2;\tilde{a}++}(\tilde{x}),C), \label{eqn:S-clear-max}
\end{eqnarray}
where $\gamma$ is an algebraic rational expression in terms of
$S_{w-2;\tilde{a}}(\tilde{x}),~S_{w-2;\tilde{a}+}(\tilde{x}),~
S_{w-2;\tilde{a}++}(\tilde{x}),$ and $C$. As messy as this is, it amounts to completing the square.
Note that for all $x,y\geq 1$, we have that $C\geq 2$ and so
$$xy + \frac{C+1 - 4S_{w-2;\tilde{a}+}(\tilde{x})/S_{w-2;\tilde{a}++}(\tilde{x})}{2}
\geq \frac{1}{2} > 0.$$ Clearly, when $C=x+y$ fixed, then
$xy + \frac{C+1 - 4S_{w-2;\tilde{a}+}(\tilde{x})/S_{w-2;\tilde{a}++}
  (\tilde{x})}{2}$ is, by AM-GM inequality,
maximized exactly when $x = y = C/2$. Moreover, replacing each $x_i$ with
an integer $n_i\geq 1$ for each $i\in\{1,\ldots,w\}$, it can be seen from
(\ref{eqn:S-clear-max}) that if $n_{w-1}+1\leq n_w$, we
 replace $(x,y) = (n_{w-1},n_w)$ by
$(n_{w-1}',n_w') = (n_{w-1}+1,n_w-1)$ and then $S_{w;\tilde{a}}(\tilde{n})$
increases, simply since $C = n_{w-1}+n_w = n_{w-1}' + n_w'$ remains unaltered and
\[
n_{w-1}'n_w' = (n_{w-1}+1)(n_w-1) = n_{w-1}n_w + n_w - n_{w-1} -1 \geq n_{w-1}n_w.
\]
By symmetry of the variables $x_1,\ldots,x_w$, we therefore obtain the following:
\begin{proposition}
  \label{prp:disc-opt}
  Let $w\geq 2$ and $\tilde{a}$ satisfy $a_{k-1}\leq a_k$ for each
  $k\in\{1,\ldots,w\}$. If $n_1,\ldots,n_w\in{\nats}$ are natural numbers
  and $\sum_in_i = n$, then the maximum value of
  $S_{w;\tilde{a}}(\tilde{n})$ from
  (\ref{eqn:Sw}) is obtained when for each $i,j\in \{1,\ldots,w\}$
  we have $|n_i - n_j|\in \{0,1\}$. In particular, for a fixed $n,w\in{\nats}$
  the $w$-tuple $\tilde{n}$ yielding the maximum value of 
  $S_{w;\tilde{a}}(\tilde{n})$ for integers $n_i\geq 1$ is unique up to reordering.
\end{proposition}
In demonstrating this, similarities are apparent to one way of proving  Tur\'{a}n's Theorem (the fifth proof in \cite{THEBOOK}), where cliques of differing sizes are made as uniform as possible,  proceeding one element at a time. A \emph{Tur\'{a}n graph} $T_k(n)$ is a complete $k$-partite graph on $n$ vertices where each anticlique (or independent set of vertices) has order as close to $\frac{n}{k}$ as possible \cite{agnarsson2006graph}. Tur\'{a}n's Theorem states that this graph has the maximal number of edges of any graph that does not contain $K_{n+1}$, the complete graph on $n+1$ vertices, as a subgraph. Figure \ref{fig:t73} depicts the Tur\'{a}n graph $T_3(7)$ on seven vertices with 16 edges, the maximum number of edges among all simple graphs on seven vertices that do not contain $K_4$ as a subgraph.
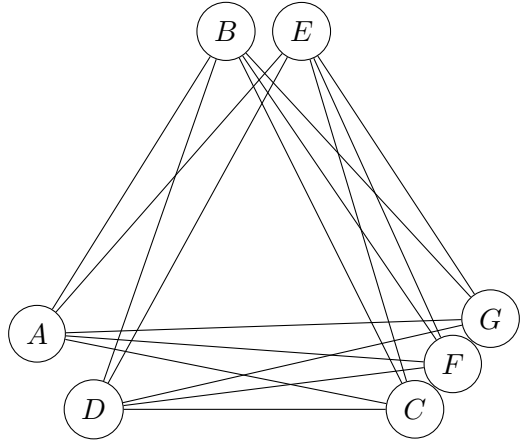
\begin{figure}[h!]
    \centering
    \begin{tikzpicture}[every  node/.style={circle,draw}]
        \node (A) at (0, 1) {$A$};
        \node (B) at (2.5, 5) {$B$};
        \node (C) at (5, 0) {$C$};
        \node (D) at (.75, 0) {$D$};
        \node (E) at (3.5, 5) {$E$};
        \node (F) at (5.5, 0.6) {$F$};
        \node (G) at (6,1.2) {$G$};
        \draw [-] (A) edge (B) (A) edge (C) (A) edge (E) (B) edge (C) (B) edge (D) (B) edge (F) (C) edge (D) (C) edge (E) (D) edge (E) (D) edge (F) (E) edge (F) (F) edge (A) (G) edge (E) (G) edge (B) (G) edge (A) (G) edge (D);
    \end{tikzpicture}
    \caption{A Tur\'{a}n graph on 7 vertices, showing the maximum number of 16 edges among simple graphs on 7 vertices without containing $K_4$ as a subgraph}
    \label{fig:t73}
\end{figure}
Motivated by this similarity, we present the following definition:
\begin{definition}
    \label{def:tur-pos}
     A \emph{Tur\'{a}n poset} on $n$ elements with width $w$, denoted $\Turan(n,w)$, consists of $w$ disjoint chains of length as close as possible to $\frac{n}{w}$. If $w \equiv r \mod n$, then $r$ of the chains are of length $\floor{\frac{n}{w}}+1$ and the remainder have length $\floor{\frac{n}{w}}$. 
\end{definition}
Then, in terms of this definition, and by Proposition~\ref{prp:disc-opt}, we obtain the following corollary:
\begin{corollary}
    \label{cor:tur-max-inter}
     A poset of cardinality $n$ consisting of $w$ disjoint chains with the maximal number of possible strict alternating cycles must be a Tur\'{a}n poset. 
\end{corollary} 
For a general poset $\Pos$, one can use Dilworth's Theorem \cite{Dilworth} to decompose into $w$ disjoint chains. In this presentation, $\Pos$ can be thought of as a disjoint union of chains $\chain(\tilde{n})=\chain_1 \cup \dots \cup \chain_w$, with the addition of any relations in $\Pos$ between elements in different chains. Thus, we can apply Proposition~\ref{prp:stict2strict} to the order preservation from 
$\chain(\tilde{n})=\chain_1 \cup \dots \cup \chain_w$ to $\Pos$ to get $|\mathcal{AC}^s(\Pos)|\leq |\mathcal{AC}^s(\chain(\tilde{n})|$.
\begin{corollary}
    \label{cor:tur-max-final}
    Consider $\Pos$ a poset of finite cardinality $n$ and fixed width $w$. If $\Turan(n,w)$ denotes a Tur\'{a}n poset of cardinality $n$ and width $w$, then $|\mathcal{AC}^s(\Pos)|\leq|\mathcal{AC}^s(\Turan(n,w))|.$
\end{corollary}

\section{Quantifying the Bound}

Recall the usual norm on the $w$-dimensional Euclidean space ${\reals}^w$
given by $$\|\tilde{x}\| = \sqrt{x_1^2 + \cdots + x_w^2}$$
where $\tilde{x} = (x_1,\ldots,x_w)\in {\reals}^w,$ and the accompanying
metric $d(\tilde{x},\tilde{y}) = \|\tilde{x} - \tilde{y}\|$. Also recall
that for any continuous function $f :{\reals}^w \rightarrow {\reals}$
and a convergent sequence $(\tilde{x}_n)_{n\geq 1}$ in $\reals^w$ with
$\tilde{x}_n \rightarrow \tilde{x}^*$ as $n\rightarrow\infty$, we have that 
$f(\tilde{x}_n)\rightarrow f(\tilde{x}^*)$ as $n\rightarrow\infty$.

Suppose $\tilde{x} = (x_1,\ldots,x_w) \in{\reals}^w$ with $x_i\geq 1$ for each $i$ and let
\begin{eqnarray*}    
\max(\tilde{x}) & :=&\max\{x_i:1\leq i \leq w\} \\
\min(\tilde{x})&:=&\min\{x_i:1\leq i \leq w\}\\
\extr(\tilde{x})& :=& \max(\tilde{x}) - \min(\tilde{x}). 
\end{eqnarray*} 
Assume further $1\leq x_1\leq\cdots\leq x_w$ and so $\extr(\tilde{x}) = x_w - x_1$.
Form another ordered tuple $\tilde{x}'$ from $\tilde{x}$ by
(i) first replacing $(x_i,x_{w-i+1})$ with $((x_i+x_{w-i+1})/2,(x_i+x_{w-i+1})/2)$
for each $i\in\{1,\ldots \lfloor w/2\rfloor$ and then (ii) permuting the
coordinates so they are in an increasing order
$x_1'\leq\cdots\leq x_w'$. Note that when $w$ is odd,
then the value of the middle coordinate of $\tilde{x}$ is left unaltered,
but might be moved to another coordinate when reordered.
\begin{claim}
  \label{clm:xx'}
  For each such $\tilde{x}\in{\reals}^w$ we
  have $\extr(\tilde{x}') \leq \extr(\tilde{x})/2$.
\end{claim}
\begin{proof}
In the case where $\tilde{x}$ is uniform, then $\tilde{x}=\tilde{x}'$ and $\extr(\tilde{x})/2=\extr(\tilde{x}')=0$. If not, first let $\overline{x}$ denote the median of $\tilde{x}$. By construction, it holds that $$\extr(\tilde{x}')= \frac{x_i+x_{w-i+1}}{2}-\frac{x_j+x_{w-j+1}}{2}$$ for some $i, j\in \{1,\ldots,\lfloor \frac{w}{2}\rfloor\}$. 
  Observe that by definition, $x_1 \leq x_j,~x_i \leq \overline{x} \leq x_{w-i+1},$ and $x_{w-j+1} \leq x_w$. Thus $x_i-x_j \leq \overline{x}-x_1$ and $x_{w-j-1}-x_{w-i-1} \leq x_w-\overline{x}$.
  Therefore: 
  \begin{eqnarray*}
        \extr(\tilde{x}') &=&\frac{x_i+x_{w-i+1}}{2}-\frac{x_j+x_{w-j+1}}{2}\\
       &= &\frac{x_i-x_j}{2}
  +\frac{x_{w-j+1}-x_{w-i+1}}{2} \\
  &\leq& \frac{\overline{x}-x_1}{2}+\frac{x_w-\overline{x}}{2} \\
  &\leq&\frac{x_w-x_1}{2} \\
  & \leq & \frac{\extr(\tilde{x})}{2}.
  \end{eqnarray*}~
\end{proof}   
Given $\tilde{x}\in {\reals}^w$, a sequence 
$(\tilde{x}_n)_{n\geq 0}$ can be formed by the recursion $\tilde{x}_0 = \tilde{x}$
and $\tilde{x}_{n+1} = \tilde{x}_n'$ from the above construction.
Note that $\frac{1}{w}\sum_{i=1}^w x_{n;i} = \mu$ is fixed for each term $\tilde{x}_n=(x_{n;1}, \dots,x_{n;w})$. 
By Claim~\ref{clm:xx'} we have
$\extr(\tilde{x}_n)\leq \extr(\tilde{x})/2^n$ for each $n$. Hence
$\extr(\tilde{x}_n)\rightarrow 0$ as $n$ tends to infinity, and therefore we have the following.
\begin{lemma}
$$
\lim_{n\rightarrow\infty}\tilde{x}_n = 
\left(\frac{\sum_{i=1}^wx_i}{w},\ldots,\frac{\sum_{i=1}^wx_i}{w}\right)=(\mu,\mu, \dots, \mu).
$$
\end{lemma}
Applying the symmetric function $S_{w;\tilde{a}}(\tilde{x})$ from
(\ref{eqn:Sw}) and (\ref{eqn:S-clear-max}) it follows that
\[
S_{w;\tilde{a}}(\tilde{x})
= S_{w;\tilde{a}}(\tilde{x}_0)
\leq S_{w;\tilde{a}}(\tilde{x}_1)
\leq\cdots\leq S_{w;\tilde{a}}(\tilde{x}_n)\leq\cdots,
\]
and hence 
\begin{equation}
  \label{eqn:real-upper-bound}
  S_{w;\tilde{a}}(\tilde{x}) \leq
  S_{w;\tilde{a}}\left(\frac{\sum_{i=1}^wx_i}{w},
  \ldots,\frac{\sum_{i=1}^wx_i}{w}\right).
\end{equation}
By the symmetry of $S_{w;\tilde{a}}(\tilde{x})$, the above
inequality (\ref{eqn:real-upper-bound}) holds for any
$\tilde{x}\in {\reals}^w$, regardless of the ordering of the coordinates. Finally, this yields a concrete
upper bound of $S_{w;\tilde{a}}(\tilde{n})$ when $\sum_in_i = n$ is fixed:
\begin{theorem}
    \label{thm:S_bound}
    If $\tilde{x}\in \mathbb{R}^w$ with $\sum_{i=1}^w x_i = n$, then 
    $S_{w;\tilde{a}}(\tilde{x}) \leq S_{w;\tilde{a}}(\frac{n}{w},\ldots,\frac{n}{w})$.
\end{theorem}
\begin{corollary}
    \label{cor:bigidea}
    Let $\Pos=(X,\leq)$ be a poset of fixed width $w$ and cardinality $|X| \leq n$. 
    Then the number of strict alternating cycles on $\Pos$, $|\mathcal{AC}^s(\Pos)|$, satisfies: 
    $$|\mathcal{AC}^s(\Pos)| \leq s_w(n):= \sum_{k=2}^w(k-1)!\binom{w}{k}\left(\frac{n(n+w)}{2w^2}\right)^{k}.$$
\end{corollary}
\begin{rmk*}
    When $n=w$, the above bound is precisely the number of strict alternating cycles of the trivial poset $\Pos_0$, as obtained in Proposition~\ref{prp:trivposet} and asymptotically satisfying Proposition~\ref{prp:alt-cycles-trivial}.
\end{rmk*}

\section{Computational Complexity and Corollaries}

Here it has been shown that if $w(n) = n$, so the poset is the trivial poset
$\Pos_0$, then the number $|\mathcal{AC}^s(\Pos_0)|$ of strict alternating cycles is asymptotically $e(n-1)!$. As evinced by Table \ref{table:1} in the appendix, anything proportional to its input (in this case $n$) factorial grows rapidly enough that computation quickly becomes infeasible.  On the other hand, if $w(n)$ is a constant, we have the following corollary:
\begin{corollary}
    \label{cor:poly2n}
    For all posets $\Pos=(X,\leq)$ with cardinality $n$ and constant width $w$, $s_w(n)$ from Corollary~\ref{cor:bigidea}  is a polynomial in $n$. Hence the maximum number of strict alternating cycles of posets of varying order $n$ and fixed width $w$ grows at worst polynomially in $n$.
\end{corollary}
   The highest order term of $s_w(n)$ is 
   $$(w-1)!\binom{w}{w}\left(\frac{n(n+w)}{2w^2}\right)^w=(w-1)!\left(\frac{n^2+nw}{2w^2}\right)^w.$$ 
   As $n \to \infty$ the linear term is negligible, so it is asymptotically  $\frac{(w-1)!}{2^ww^{2w}} n^{2w}$, which is $O(n^{2w})$. 
   
   Bridging this gap is an area of interest, and results to that effect are as follows:
\begin{lemma}
    \label{lmm:ass-alt-cycles}
    For all posets $\Pos=(X,\leq)$ of cardinality $|X| = n$ with width of $w = w(n)$, an increasing function of $n$ that we 
    will assume tends to infinity when $n$ does, 
    $s_w(n)$ from Corollary~\ref{cor:bigidea} satisfies 
    \[
    s_w(n) = \Theta\left((w-1)!\left(\frac{n(n+w)}{2w^2}\right)^w\right) = 
    \Theta\left(\sqrt{w-1}\left(\frac{w-1}{e}\right)^{w-1}\left(\frac{n(n+w)}{2w^2}\right)^w\right),
    \]
    as $n$ tends to infinity.
\end{lemma}
\begin{proof}
Let 
\[
s_w(n) = \sum_{k=2}^w(k-1)!\binom{w} {k}\left(\frac{n(n+w)}{2w^2}\right)^{k},
\]
as in Corollary~\ref{cor:bigidea}.
First, note that $s_w(n)$ is strictly greater than the last summand when $k=w$, and so 
\[
s_w(n) \geq (w-1)!\left(\frac{n(n+w)}{2w^2}\right)^w.
\]
Secondly, since $w\leq n$ then $n(n+w)/(2w^2)\geq 1$ and 
therefore the following holds:
\begin{eqnarray*}
s_w(n) &= &
w!\left(\frac{n(n+w)}{2w^2}\right)^w\sum_{k=2}^w\frac{1}{(w-k)!k}\left(\frac{n(n+w)}{2w^2}\right)^{k-w} \\
&\leq&
w!\left(\frac{n(n+w)}{2w^2}\right)^w\sum_{k=2}^w\frac{1}{(w-k)!k}.
\end{eqnarray*}
By replacing $n$ in $g(n)$ from Proposition~\ref{prp:trivposet} with $w$ to obtain
$$
g(w) = \sum_{k=0}^{w-2}\frac{1}{k!}\frac{1}{w-k}=\sum_{k=2}^{w} \frac{1}{(w-k)!}\frac{1}{k},
$$
and since $\lim_{n\rightarrow \infty}w(n) = \infty$, then for sufficiently large $w$ we can apply (\ref{eqn:gn-concr}) to obtain that
$$
g(w) = \sum_{k=2}^w\frac{1}{(w-k)!k}\leq \frac{3}{w}
$$ for large enough $w$, and hence 
\[
s_w(n) \leq w!\left(\frac{n(n+w)}{2w^2}\right)^w \cdot\frac{3}{w} = 3(w-1)!\left(\frac{n(n+w)}{2w^2}\right)^w.
\]
Applying Stirling's approximation formula (see \cite{Graham-Knuth-Patashnik}) to $(w-1)!$ we obtain 
the lemma.
\end{proof}
Using Lemma~\ref{lmm:ass-alt-cycles} will yield four corollaries, starting with:
\begin{corollary}
    \label{cor:sublinear}
    If $\frac{w(n)}{n}\to 0$ as $n \to \infty$, then for sufficiently large $n$:
     $$\log(s_w(n))=\Theta(w \log (n)).$$
\end{corollary}
\begin{proof}
    From Lemma~\ref{lmm:ass-alt-cycles}, $s_w(n) \leq C(w-1)!\left(\frac{n(n+w)}{2w^2}\right)^w$, where $C$ is a positive constant (in fact $C=3$ 
    works as demonstrated in the previous proof). Multiplying by $w$ and applying Stirling's approximation yields
    $$ws_w(n) \leq  C\sqrt{w}\left(\frac{w}{e}\right)^w\left( \frac{n(n+w)}{2w^2}\right)^w=C\sqrt{w}\left(\frac{n(n+w)}{2ew}\right)^w.$$
    Thus $s_w(n)\leq \frac{C}{\sqrt{w}}\left(\frac{n(n+w)}{2ew}\right)^w$. Applying $\log$ to both sides:
    $$
        \log(s_w(n))  \leq  \log(C) -\frac{1}{2}\log(w)+w\log(n^2+nw)-w\log(2ew)
    $$
    The value $\log(C)- \frac{1}{2}\log(w)-w\log(2ew)$ is certainly less than $w\log(n)$ for large $n$, and $$w\log(n^2+nw) \leq w\log(2n^2)\leq 4w\log(n).$$
    Thus $\log(s_w(n) \leq 5w \log(n) $.
    
    A similar argument using the lower bound from Lemma~\ref{lmm:ass-alt-cycles} shows that for some constant $D$ we have:
    $$\log(s_w(n)) \geq \log(D) -\frac{1}{2}\log(w)+w\log(n^2+nw)-w\log(2ew).$$
    Since $w\log(n^2) \leq w\log(n^2+nw)$ and $-\frac{3}{2}w\log(n) \leq \log(D)-\frac{1}{2}\log(w)-w\log(2ew)$ for large enough $n$, we can conclude that $s_w(n)\geq \frac{1}{2}w\log(n)$ and therefore $\log(s_w(n))=\Theta(w\log(n)).$
\end{proof}
This will make two cases of interest straightforward to consider. The first is as follows: 
\begin{corollary}
    \label{cor:expalpha}
    For $w(n)=Cn^\alpha$, with $C$ a constant and $\alpha$ a real constant with $0<\alpha<1$, we have 
    $s_w(n)=n^{\Theta(n^{\alpha})}$, which is greater than quasi-polynomial but less than exponential in $n$. 
\end{corollary}
\begin{proof}
    Using Corollary~\ref{cor:sublinear}, we plug in $Cn^\alpha$ for $w$ to yield
    $$C_1n^a\log(n)\leq\log(s_w(n)) \leq C_2n^\alpha \log(n)).$$ Exponentiating all terms will yield 
    $$n^{C_1 n^\alpha} \leq s_w(n)\leq n^{C_2n^\alpha}$$
    which is the desired result.
\end{proof}
Additionally, we have the following:
\begin{corollary}
    \label{cor:polylog}
    If $w(n)=C\log^k(n)$, with $C,k \in \reals^+$ constants, then $\log(s_w(n))=O(\log^{k+1}n)$, and therefore $s_w(n)$ is quasi-polynomial in $n$.
\end{corollary}
\begin{proof}
    Plugging in to Corollary~\ref{cor:sublinear}:
    $$\log(s_w(n))\leq C \log^k(n)\log(n)=C\log^{k+1}(n)$$
    Therefore $s_w(n)\leq Ce^{\log^{k+1}(n)}.$
\end{proof}
Recall the Gamma function $\Gamma(z):= \int_0^\infty t^{z-1}e^{-t}dt$ for $z$ a complex value with positive real component, and its property that if $z=x \in \reals^+$ then $\Gamma(x+1)=x\Gamma(x)$. 
\begin{corollary}
    \label{cor:alphan}
    If $w$ is a fixed proportion of $n$, such that $w=\alpha n$ with $0<\alpha <1$, then 
    $$
    s_w(n)=\Theta \left(\Gamma(\alpha n) \beta^n \right)
    $$  
    for some positive real constant $\beta$.
\end{corollary}
\begin{proof}
    By Lemma \ref{lmm:ass-alt-cycles} we have that 
    $$s_w(n)=\Theta\left((w-1)!\left(\frac{n(n+w)}{2w^2}\right)^w\right) $$
    Replacing each $w$ with $\alpha n$ and canceling $w^2$ from the fractional term yields:
    $$\Theta\left(\Gamma\left(\alpha n\right)\left(\frac{1+\alpha}{2\alpha^2}\right)^{\alpha n}\right).$$
    Letting $\beta=\left(\frac{1+\alpha}{2\alpha^2}\right)^\alpha$ proves the corollary. 
\end{proof}
As a function of $\alpha,$ $\beta$ is notably, for positive real input, at most roughly $1.82$ for an input $\alpha$ of roughly $.34$. These values were derived numerically: the maximum occurs at precisely at the positive root of $f(x)=2 + x + \log(2) + x \log(2) - (1 + x) \log(\frac{1 + x}{x^2})$, and the precise value it attains is even more elaborate. Figure \ref{fig:betagraph} shows this behavior.
\begin{figure}[h!]
    \centering
    \tikz[scale=.8,domain=0:1.1,samples=50]{
    \begin{axis}[axis x line=middle, axis y line=middle, stack plots=x, xmin=0, xmax=1,
    ymin=0, ymax=2]
        \addplot+[mark=none, color=black] {((x+1)/(2*x^2))^x} \closedcycle;
      
        \node[circle,fill,scale=.3] at (axis cs:.34,1.82){};
        \node[anchor=south,scale=1] at (axis cs:.34,1.82){$\sim(.34,1.82)$};
    
    \end{axis}
    \node[anchor=south west,scale=1.5] at (current axis.right of origin) {$\alpha$};
    \node[anchor=north west,scale=1.5] at (current axis.above origin) {$\beta$};
}
    \caption{Behavior of $(\frac{1+\alpha}{2\alpha^2})^\alpha = \beta$}
    \label{fig:betagraph}
\end{figure}
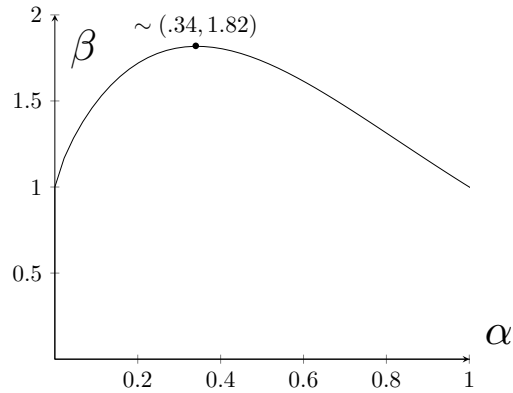
 Ultimately this means there is some interesting behavior in $\beta^n$ that is being masked by the dominance of the factorial term.

\section{Summary}

The main results of this article are Proposition \ref{prp:max-trivial}, Proposition \ref{prp:disc-opt}, and Corollary \ref{cor:bigidea}, establishing upper bounds on the number of alternating cycles possible on a general poset $\Pos$ (and an asymptotic approximation of this value) and a poset of fixed width $\Pos_w$ respectively. The general poset is shown to be bounded by the case of the trivial poset, and a poset of fixed width $w$ bounded by a Tur\'{a}n poset (Definition \ref{def:tur-pos}) of width $w$. Some corollaries regarding growth of $s_w(n)$ depending on $w=w(n)$, where $w(n)$ is an increasing function of $n$, are also shown. By the nature of strict alternating cycles, these bounds also apply to the number of hyperedges of the hypergraph $\mathcal{H}^s(\Pos)$ associated to $\Pos$, which has chromatic number $\odim(\Pos)$. 



%
\section{Supplementary Content}
\begin{table}[h!]
\centering
\begin{tabular}{||c c c c||} 
 \hline
 $n$ & Alternating Cycles on $\Pos_0$ with $|X|=n$ & $\sim e(n-1)!$ & Undercount (percent) \\ [0.5ex] 
 \hline\hline
 2 & 1 & e & -171.82\% \\ 
 3 & 5 & 5.4366 & -8.73\% \\
 3 & 20 & 16.310 & 18.45\% \\
 4 & 84 & 65.239 & 22.33\% \\
 6 & 409 & 326.19 & 20.25\% \\
 7& 2365& 1957.2& 17.24\%\\
 8& 16064& $1.37\times 10^4$& 14.71\% \\
 9 & 125664&$1.096\times10^5$ & 12.78\%\\
 10 & 1112073&$9.8641\times10^6$ & 11.3\%\\
 11 & 10976173& $9.8641\times10^6$& 10.13\%\\
 12 & 119481284& $1.0851\times10^8$& 9.19\% \\
 13 & 1421642628& $1.3021\times10^9$&  8.41\% \\
 14 & 18348340113& $1.6927\times 10^{10}$& 7.75\%\\
 15 & 255323504917& $2.3698\times 10^{11}$ &7.19\% \\
 16 & 3809950976992& $3.5546\times10^{12}$ & 6.70\% \\
 17 & 60683990530208& $5.6874\times 10^{13}$ & 6.28\%\\
 18 & 1027542662934897& $9.6686\times 10^{14}$& 5.91\% \\
 19 & 18430998766219317& $1.7403\times 10^{16}$ & 5.57\% \\
 20 & 349096664728623316& $3.3067\times 10^{17}$& 5.28\% \\
 21 & 6962409983976703316 & $6.6133\times 10^{18}$ & 5.01\%\\
 22 & 145841989688186383337 & $1.3888 \times 10^{20}$& 4.77\%\\
 23 & 3201192743180799343821 & $3.0554 \times 10^{21}$ & 4.55\%\\
 24 & 73474260073510897434976 & $7.0273 \times 10^{22}$ & 4.36\%\\
 25 & 1760027876001433251622720 & $1.6866 \times 10^{24}$ & 4.17 \%\\
 50 &$1.6879\times 10^{63}$ & $1.6535\times 10^{63}$ & 2.04\% \\
 100& $2.5628 \times 10^{156}$ &$2.5369\times 10^{156}$ & 1.01\% \\[1ex] 
 \hline
\end{tabular}
\caption{Sample Values Related to $|\mathcal{AC}^s(\Pos_0)|$ (and therefore $\mathcal{MAC}^s(n)$)}
\label{table:1}
\end{table}
\newpage

\bibliographystyle{amsalpha}
\bibliography{gmuETD}
\end{document}